\documentclass[12pt]{article}
\usepackage{amsmath,amsfonts,amssymb,amsthm,amscd}
\title{The stable regularity lemma revisited}
\date{\today}
\author{Maryanthe Malliaris\thanks{Partially supported by NSF grant DMS-1300634 and a Sloan fellowship  }\\University of Chicago \and Anand Pillay\thanks{Partially supported by NSF grant DMS-1360702}\\University of Notre Dame}
\newtheorem{Theorem}{Theorem}[section]

\newtheorem{Remark}[Theorem]{Remark}
\newtheorem{Lemma}[Theorem]{Lemma}

\newtheorem{Fact}[Theorem]{Fact}

\begin{document}
\maketitle

\begin{abstract} We prove a regularity lemma  with respect to  arbitrary Keisler measures $\mu$ on $V$, $\nu$ on $W$ where the bipartite graph $(V,W,R)$ is definable in a saturated structure ${\bar M}$ and the formula $R(x,y)$ is stable.  The proof is rather quick, making use of  local stability theory. The special case where $(V,W,R)$ is pseudofinite, $\mu$, $\nu$ are the counting measures, and ${\bar M}$ is suitably chosen (for example a nonstandard model of set theory), yields the stable regularity theorem of  \cite{MS}, though without explicit bounds or equitability. 
\end{abstract}

\section{Introduction}
We refer to \cite{MS} for a discussion of Szemer\'edi's Regularity Lemma and various elaborations on it. 
Our context  is a saturated model ${\bar M}$ of an arbitrary theory,   and a definable (with parameters) bipartite graph $(V,W,R)$ where the edge relation $R(x,y)$ is  stable. Stability of $R$ means that for some $k$ there do not exist $a_{i}\in V$, $b_{i}\in W$ for $i\leq k$ such that $R(a_{i},b_{j})$ holds iff $i\leq j$.  We also have global Keisler measures $\mu_{x}$ on $V$ and $\nu_{y}$ on $W$ (i.e. finitely additive probability measures on the Boolean algebras of definable subsets of $V,W$ respectively).  
By a {\em $\Delta$-formula} we mean a finite Boolean combination of things like $R(x,b)$ and $x=a$, and by a $\Delta^{*}(x,y)$-formula we mean a finite Boolean combination of things like $R(a,y)$ and $y = b$.  We don't actually need the $=$ formulas but it is convenient to include them. By convention the $x$ variable is restricted to $V$ and $y$ variable to $W$ so by definition a $\Delta$-formula ($\Delta^{*}$-formula)  defines a subset of $V$ ($W$).  

\begin{Theorem}  In the context above, namely a definable graph $(V,W,R)$ in a saturated structure ${\bar M}$, where $R(x,y)$ is stable, and Keisler measures $\mu$ on $V$, $\nu$ on $W$, we have the following:  Given $\epsilon > 0$,   we can partition $V$ into finitely many definable sets $V_{1},..,V_{m}$ each defined by a $\Delta$-formula,  and also partition $W$ into finitely many  $W_{1},..,W_{m}$ each defined by a $\Delta^{*}$-formula, such that for each $V_{i}, W_{j}$ exactly one of the following holds:
\newline
(i)  for all $a\in V_{i}$ outside a set of $\mu$ measure $\leq \epsilon \mu(V_{i})$, for all $b\in W_{j}$ outside a set of
$\nu$ measure $\leq \epsilon \nu(W_{j})$, we have $R(a,b)$, AND DUALLY for all $b\in W_{j}$ outside a set of
$\nu$ measure $\leq \epsilon \nu(W_{j})$, for all $a\in V_{i}$ outside a set of $\mu$ measure $\leq \epsilon \mu(V_{i})$,  we have $R(a,b)$ holds, or
\newline
(ii) for all $a\in V_{i}$ outside a set of $\mu$ measure $\leq \epsilon \mu(V_{i})$, for all $b\in W_{j}$ outside a subset of
$\nu$ measure $\leq \epsilon \nu(W_{j})$, we have $\neg R(a,b)$, AND DUALLY for all $b\in W_{j}$ outside a subset of
$\nu$ measure $\leq \epsilon \nu(W_{j})$, for all $a\in V_{i}$ outside a set of $\mu$ measure $\leq \epsilon \mu(V_{i})$, $\neg R(a,b)$. 
\end{Theorem}

\begin{Remark} We obtain $\delta$-regularity for suitable $\delta$ and in a suitable sense, of each of the $(V_{i},W_{j},R|(V_{i}\times W_{j}))$ (as essentially in Claim 5.17 of \cite{MS}):  Take 
$\delta = \sqrt{2\epsilon}$, and suppose we are in case (i) of the conclusion of the theorem. Suppose $A\subseteq V_{i}$ is definable (not necessarily 
by a $\Delta$-formula), and $B\subseteq W_{j}$ is definable (not necessarily by a $\Delta^{*}$-formula), with $\mu(A)\geq \delta\mu(V_{i})$ and 
$\nu(B) \geq \delta\nu(W_{j})$. Then  for all $a\in A$ outside a set of $\mu$-measure $\leq \delta\mu(A)$, for all $b\in B$ outside a set of $\nu$-measure $\leq \delta\nu(B)$ we have $R(a,b)$. 
\end{Remark}

\begin{Remark} We also recover the regularity lemma for finite stable graphs in the sense of \cite{MS} by considering the counting measure on  pseudofinite graphs. (But without explicit  bounds or equitability of the partition.)
\end{Remark}
\noindent
{\em Explanation.}  Let $(V_{i},W_{i}, R_{i})$ be a family of finite graphs where the relation $R_{i}$ does not have the $k$-order property for fixed $k$. Take an infinite ultraproduct, to obtain a definable nonstandard finite graph $(V,W,R)$ in a (saturated) nonstandard model of set theory ${\bar M}$. Then $R(x,y)$ is stable. 
Let $\mu$ be the counting measure on $V$, namely for definable (in ${\bar M})$ $A\subseteq V$, define $\mu(A)$ to be the standard part of $|A|/|V|$, and likewise for $\nu$ on $W$. Then $\mu$, $\nu$ are global Keisler measures and we can apply the theorem.  Note that the definable sets in ${\bar M}$ are the internal sets in the sense of nonstandard analysis, so the standard compactness arguments give regularity in the finite situation with respect to arbitrary subsets (remembering the parenthesis). 

\vspace{2mm}
\noindent
Our model theory notation is standard. Stability theory is due to Shelah and the basic reference is \cite{Shelah}. For the purposes of the current paper it is convenient to refer to Chapter 1 of \cite{Pillay}. We will be in particular making use of ``local stability theory" namely the theory of forking for complete $\Delta$-type $p(x)$ where $\Delta$ is a finite set of stable formulas $\phi(x,y)$, and where a complete $\Delta$-type over a model $M$ is given (axiomatized)  by a maximal consistent set of formulas of the form $\phi(x,b)$, $\neg \phi(x,b)$ for $\phi(x,y)\in \Delta$ and $b\in M$.  The $\Delta$-rank refers to $R_{\Delta}(-)$ as in Section 1.3 of \cite{Pillay}, which is denoted $R(-,\Delta, \omega)$ in \cite{Shelah}. By a $\Delta$-formula we mean a finite Boolean combination of formulas $\phi(x,b)$ for $\phi\in \Delta$.   Recall that a formula $\phi(x,b)$ divides over $A$ if there an $A$-indiscernible sequence $b = b_{0}, b_{1}, b_{2}...$ such that $\{\phi(x,b_{i}):i < \omega\}$ is inconsistent. As in this context we will be concerned with stable formulas, we use the expression ``forking" to mean dividing.

Keisler measures are of course important in this paper. If ${\bar M}$ is a saturated model and $V$ a sort, a Keisler measure $\mu$ on $V$ is a finitely additive probability measure on the Boolean algebra of definable (with parameters) subsets of $V$. If we are only concerned with subsets definable over $M$ we call $\mu$ a Keisler measure on $V$ over $M$.  If $\Delta$ is a collection of $L$-formulas $\phi(x,y)$ ($x$ fixed of sort $V$, $y$ arbitrary), we can restrict $\mu$ to $\Delta$  formulas and call it $\mu|\Delta$. Likewise $\mu|(\Delta,M)$ is the restriction of $\mu$ to $\Delta$-formulas over $M$.  We will  use freely facts (from \cite{Keisler}) such that if $\mu$ is a global Keisler measure on $V$ say, then for any type-definable subset $X$ of $V$, type-defined over a small set of parameters (equivalently any partial type $\Sigma$ over a small set of parameters), $\mu(X)$ is defined and is approximated from above  by the $\mu$- measure of definable sets (formulas in $\Sigma$).

We depend  somewhat on Keisler's seminal paper \cite{Keisler}, which is very much concerned with Keisler measures  $\mu|\Delta$ for $\Delta$ a finite set of stable formulas. A key observation (Lemma 1.7 of \cite{Keisler}), is that such $\mu|\Delta$ is a weighted sum of complete $\Delta$-types (the proof of which  is more or less repeated in Lemma 2.1 in the next section).

 We will need the following, which  after a suitable translation is Proposition 1.20 of \cite{Keisler}: 
\begin{Fact} Suppose $\Delta$ is a finite set of stable formulas and $\mu_{x}$ is a global Keisler measure. Then there is a small model $M$ such that $\mu|\Delta$ does not fork over $M$ in the sense that each $\Delta$-formula $\psi(x)$ of positive $\mu$-measure does not fork over $M$.
\end{Fact}

In fact $\mu|\Delta$ will be the unique nonforking extension of $\mu|(\Delta,M)$, so $\mu|\Delta$ will be a ``generically stable" measure in the sense of \cite{HPS}. But all we will really need is the \begin{Fact} A complete $\Delta$-type over a model $M$ has a unique nonforking extension to a global complete $\Delta$-type. 
\end{Fact}

\vspace{2mm}
\noindent
As in the statement of the theorem we will take $\Delta$ to be $\{R(x,y),x=z\}$ and likewise for $\Delta^{*}$.

\vspace{5mm}
\noindent
{\em Acknowledgement.} Thanks to Sergei Starchenko for comments on a first draft of the paper.

\section{Proof of Theorem 1.1}
First by Fact 1.4, let $M$ be a small model such that $(V,W,R)$ is definable over $M$, and both $\mu|\Delta$, $\nu|\Delta^{*}$ do not fork over $M$. 

\begin{Lemma}  Given $\epsilon > 0$, we can write $V$ as a disjoint union of $\Delta$-formulas (restricted to $V$) over $M$,  $V = V_{1}\cup ... \cup V_{m}$, such that for each $i$ there is a complete $\Delta$-type $p_{i}$ over $M$ such that $\mu(p_{i}) > 0$, $V_{i}\in p_{i}$ (or $p_{i}\subseteq V_{i}$)  and $\mu(V_{i}\setminus p_{i}) \leq \epsilon\mu(V_{i})$.

\end{Lemma}
\begin{proof} By induction on $R_{\Delta}(V)$ (which is a finite number). If it is $0$, then $V$ is finite and consists of finitely many types, all  realized in $M$, so there is nothing to do. Suppose $R_{\Delta}(V) = n > 0$. Let $p_{1},..,p_{k}$ be the complete $\Delta$-types over $M$ of $R_{\Delta} = n$.  Let $\mu(p_{i}) = \alpha_{i}$. We divide into cases.
\newline
{\em Case (i).}  $\alpha_{i} > 0$ for $i=1,..,k$. 
\newline
Then we can find $V_{i}\in p_{i}$ for $i=1,..,k$ such that $\mu(V_{i}) \leq \alpha_{i}/(1-\epsilon)$ whereby $\mu(V_{i}\setminus p_{i})\leq \epsilon\mu(V_{i})$.  We may assume that the $V_{i}$ are disjoint. Let $U = V_{1}\cup ..\cup V_{k}$. Then $V\setminus U$ has $\Delta$-rank $< n$ so we can apply the induction hypothesis to it, to obtain the lemma for $V$.  (Noting that if $\mu(V\setminus U) = 0$ we can just adjoin it to one of the $V_{i}$.) 

\vspace{2mm}
\noindent
{\em Case (ii).} Some but not all of the $\alpha_{i}$ are $0$.
\newline
Without loss $\alpha_{1},..\alpha_{\ell} > 0$ and $\alpha_{j} = 0$ for $j = \ell + 1,.., k$ where $1\leq \ell < k$.   As in Case (i), find pairwise disjoint $V_{i}\in p_{i}$ for $i = 1,.., \ell$, such that $\mu(V_{i}\setminus p_{i}) \leq \mu(V_{i})$.  Now as $\mu(V_{j}) = 0$ for $j=\ell+1,..,k$ we can find a $\Delta$-definable over $M$ set $V'$ containing $p_{\ell+1}\cup ... 
\cup p_{k}$  and disjoint from $V_{1}\cup ..\cup V_{\ell}$ such that 
$\mu(V') \leq (\epsilon\mu(V_{1})-\mu(V_{1}\setminus p_{1}))/(1-\epsilon)$.  Put $V_{1}' = V_{1}\cup V'$. Then $\mu(V_{1}'\setminus p_{1}) \leq \epsilon\mu(V_{1}')$. Put again $U = 
V_{1}'\cup V_{2} \cup...\cup V_{\ell}$. Then $V\setminus U$ has $\Delta$-rank $< n$ and we can apply induction. 

\vspace{2mm}
\noindent
{\em Case (iii).}  All $\alpha_{i}$ are $0$. 
\newline
By  Lemma 1.7 of \cite{Keisler}, let $p$  be a complete $\Delta$-type over $M$ with $\mu(p) > 0$. So by our case hypothesis, the $\Delta$-rank of $p$ equals $m < n$. Let $W$ be a formula 
in $p$ of $\Delta$-rank $m$.  So $\mu(W) > 0$ and by induction hypothesis we find complete $\Delta$-types $q_{1},..,q_{r}$over $M$ of $\mu$-measure $> 0$,  and pairwise disjoint $W_{i}\in 
q_{i}$ such that $W = W_{1}\cup .. \cup W_{r}$ and $\mu(W_{i}\setminus q_{i}) \leq \mu(W_{i})$ for $i=1,..,r$.  From our case hypothesis we can  find a $\Delta$-formula over $M$, $V'$ 
which contains $p_{1}\cup ....\cup p_{k}$ and is disjoint from $W$ such that   $\mu(V') \leq (\epsilon\mu(W_{1})-\mu(W_{1}\setminus q_{1}))/(1-\epsilon)$.  Put $W_{1}' = W_{1}\cup V'$ and as above
$\mu(W_{1}'\setminus q_{1}) \leq \epsilon\mu(W_{1}')$.  Again let $U = W_{1}'\cup W_{2} \cup .. \cup W_{r}$. Then the $\Delta$-rank of $V\setminus U$ is $< n$ so we can apply induction to complete the proof in case (iii). 

\end{proof}

We can do the same thing for $W$ to write $W$ as a disjoint union of $\Delta^{*}$-definable (over $M$)  sets such that for each $j$ there is a complete
$\Delta^{*}$-type $q_{j}$ over $M$ such that $\nu(W_{j}\setminus q_{j}) \leq \epsilon\nu(W_{j})$.

\begin{Lemma} For each $i=1,..,m$ and $j=1,..,n$, we have (i) or (ii) of Theorem 1.1
\end{Lemma} 
\begin{proof} Fix $i,j$ and we have complete $\Delta$-type $p_{i}(x)$ over $M$ and complete $\Delta^{*}$-type $q_{j}$ over $M$. Now $p_{i}$ is definable, and  its $R$-definition is given by a $\Delta^{*}$-formula (1.27 of \cite{Pillay}). Namely there is a $\Delta^{*}$-formula $\psi(y)$ over $M$ such that for $b\in W(M)$, $R(x,b)\in p_{i}(x)$ iff $\models \psi(b)$.  Likewise if $p_{i}'$ is the unique nonforking  extension of $p_{i}$ to a complete global $\Delta$-type, $\psi(y)$ is the $R(x,y)$-definition of $p_{i}'$.  We have two cases:
\newline
Case (i):  $\psi(y)\in q_{j}$.
\newline
Hence for all $b\in W_{j}$ other than a set a set of $\nu$-measure $\leq \epsilon \nu(W_{j})$ we have $\models \psi(b)$. 
Now suppose that $\models \psi(b)$, hence $R(x,b)\in p'$, whereby $p\cup \{\neg R(x,b)\}$ divides over $M$ (by Fact 1.5) so as $\mu$ does not divide over $M$,  $\mu(p\cup\{\neg(R(x,b)\}) = 0$, so for all $a\in V_{j}$ outside a set of $\mu$-measure $\leq \epsilon \mu(V_{i})$, we have $\models R(a,b)$.  We have actually proved the second clause of  (i) of the theorem. To obtain the first clause, let $\chi(x)$ be the $\Delta^{*}$-definition of $q_{j}$, so by Lemma 2.8 of \cite{Pillay}, and our case hypothesis, $\chi(x)$ (a $\Delta$-formula over $M$) is in $p_{i}$. Continue as above.

\vspace{2mm}
\noindent
Case (ii): $\neg\psi(y)\in q_{j}$. In which case we obtain as in Case (i) that for all $a\in V_{i}$ outside a set of $\mu$-measure $\leq \epsilon\mu(V_{i})$, for all $b\in W_{j}$ outside a set of $\nu$-measure $\leq \epsilon \nu(W_{j})$ we have $\neg R(a,b)$.  

This completes the proof.

\end{proof} 

Theorem 1.1 follows directly from Lemmas 2.1 and 2.2.

\end{document}